\documentclass[12pt]{amsart}

\newtheorem{theorem}{Theorem}[section]
\newtheorem{proposition}[theorem]{Proposition}
\newtheorem{lemma}[theorem]{Lemma}
\newtheorem{corollary}[theorem]{Corollary}

\newtheorem{remark}{Remark}[section]
\newtheorem{definition}{Definition}[section]
\newtheorem{maintheorem}{Theorem}

\newtheorem{claim}{Claim}[section]

\newcommand{\blanksquare}{\,\,\,$\sqcup\!\!\!\!\sqcap$}

\newcounter{example}
{\par}

\usepackage{amscd}
\usepackage{threeparttable}
\usepackage{amssymb}
\usepackage{amsfonts}
\usepackage{epsfig}
\usepackage{graphicx}
\usepackage{amsmath}
\usepackage[all, 2cell]{xy} \UseAllTwocells \SilentMatrices
\usepackage{indentfirst}

\begin{document}

\title[On the stability of the set of hyperbolic closed orbits]{\textbf{On the stability of the set of hyperbolic closed orbits of a Hamiltonian}}

\author[M. Bessa]{M\'{a}rio Bessa}
\address{Departamento de Matem\'atica Pura, Universidade do Porto, 
Rua do Campo Alegre, 687, 
4169-007 Porto, Portugal \\ ESTGOH-Instituto Polit\'ecnico de Coimbra, Rua General Santos Costa, 3400-124 Oliveira do Hospital, Portugal}
\email{bessa@fc.up.pt}
\author[C. Ferreira]{C\'elia Ferreira}
\address{Departamento de Matem\'atica Pura, Universidade do Porto, 
Rua do Campo Alegre, 687, 
4169-007 Porto, Portugal}
\email{celiam@fc.up.pt}
\author[J. Rocha]{Jorge Rocha}
\address{Departamento de Matem\'atica Pura, Universidade do Porto, 
Rua do Campo Alegre, 687, 
4169-007 Porto, Portugal}
\email{jrocha@fc.up.pt}

\begin{abstract}

A Hamiltonian level, say a pair $(H,e)$ of a Ha\-mil\-to\-nian $H$ and an energy $e \in \mathbb{R}$, is said to be Anosov if there exists a connected component $\mathcal{E}_{H,e}$ of $H^{-1}(\left\{e\right\})$ which is uniformly hyperbolic for the Hamiltonian flow $X_H^t$. The pair $(H,e)$ is said to be a Hamiltonian star system if there exists a connected component $\mathcal{E}^\star_{H,e}$ of the energy level $H^{-1}(\left\{{e}\right\})$ such that all the closed orbits and all the critical points of $\mathcal{E}^\star_{H,e}$ are hyperbolic, and the same holds for a connected component of the energy level $\tilde{H}^{-1}(\left\{\tilde{e}\right\})$, close to $\mathcal{E}^\star_{H,e}$, for any Hamiltonian $\tilde{H}$, in some $C^2$-neighbourhood of $H$, and $\tilde{e}$ in some neighbourhood of $e$.

In this article we prove that for any four-dimensional Ha\-mil\-to\-nian star level $(H,e)$ if the surface $\mathcal{E}^\star_{H,e}$ does not contain critical points, then $X_H^t|_{\mathcal{E}^\star_{H,e}}$ is Anosov; if $\mathcal{E}^\star_{H,e}$ has critical points, then there exists $\tilde{e}$, arbitrarily close to $e$, such that $X_H^t|_{\mathcal{E}^\star_{H,\tilde{e}}}$ is Anosov.
\end{abstract}

\maketitle


%%%%%%%%%%%%%%%%%%%%%%%%%%%%%%%%%%%%%%%%%%%%%%%%%%%%%%%%%%%%%%%%%%%%%%%%%%%%%%%%%%%%%%%%%%%%%%%%%%%%%%%%%%%%%%%%%%%%%%%%%%%%%%%%%%%%%%%%%%%%%%%%%%%
%%%%%%%%%%%%%%%%%%%%%%%%%%%%%%%%%%%%%%%%%%%%%%%%%%%%%%%%%%%%%%%%%%%%%%%%%%%%%%%%%%%%%%%%%%%%%%%%%%%%%%%%%%%%%%%%%%%%%%%%%%%%%%%%%%%%%%%%%%%%%%%%%%%
%%%%%%%%%%%%%%%%%%%%%%%%%%%%%%%%%%%%%%%%%%%%%%%%%%%%%%%%%%%%%%%%%%%%%%%%%%%%%%%%%%%%%%%%%%%%%%%%%%%%%%%%%%%%%%%%%%%%%%%%%%%%%%%%%%%%%%%%%%%%%%%%%%%

\begin{section}{Introduction}

%explicar o k se faz (apontar para def em sec�oes a frente. ex: def de G^2)

%falar do teorema da estabilidade estrutural...?

Let $S$ be a dynamical system defined in a closed manifold. Roughly speaking $C^r$-structural stability ($r\geq 1$) of a dynamical system means that there exists a $C^r$-neighbourhood $\mathcal{U}$ of $S$ such that any other system in $\mathcal{U}$ is topological conjugated to $S$.  These conjugations are defined in  sets where the dynamics is relevant, usually in its nonwandering set, $\Omega(S)$. We recall that $\Omega(S)$ is the set of points in the manifold such that, for every neighbourhood $U$, there exists an iterate $n$ sa\-tis\-fying $S^n(U)\cap U\not=\emptyset$.

The notion of structural stability was first introduced in the mid 1930' by Andronov and Pontrjagin (\cite{AP}) and this concept is intrinsically related to uniform hyperbolicity (see Section~\ref{hyper} for the definition of hyperbolicity).

We say that $S$ satisfy the \emph{Axiom A} if the closure of its closed orbits  is equal to $\Omega(S)$ and, moreover, this set is hyperbolic. One of the most challenging problems in the modern theory of dynamical systems is to know if a $C^r$-structural stable system satisfy the Axiom A property. A cornerstone to this program was the remarkable proof done by Ma\~n\'e of the stability conjecture for the case of $C^1$-dissipative diffeomorphisms (\cite{Ma1}).

Back to the early 1980', Ma\~n\'e defined a set $\mathcal{F}^1$, of dissipative diffeomorphisms having a $C^1$-neighbourhood $\mathcal{U}$ such that every diffeomorphism inside $\mathcal{U}$ has all periodic orbits of hyperbolic type. In ~\cite{Ma2}, Ma\~n\'e proved that every surface dissipative diffeomorphism of $\mathcal{F}^1$ satisfies the Axiom A. Hayashi (\cite{Hay}) extended this result for higher dimensions. The set $\mathcal{F}^1$ is related to structural stability since the proof that $C^1$-structural stable system satisfies the Axiom A property mainly uses the fact that the system is in $\mathcal{F}^1$.

Recall that, by the spectral decomposition of an Axiom A system $S$, we have that $\Omega(S)=\cup_{i=1}^k \Lambda_i$ where each $\Lambda_i$ is a basic piece. We define an order relation by $\Lambda_i\prec \Lambda_j$ if there exists $x$ (outside $\Lambda_i\cup\Lambda_j$) such that $\alpha(x)\subset \Lambda_i$ and $\omega(x)\subset\Lambda_j$. We say that $S$ has a \emph{cycle} if there exists a cycle with respect to $\prec$ (see ~\cite{Shub} for details).

In fact, the mentioned results by Ma\~n\'e and Hayashi guarantee that diffeomorphisms in $\mathcal{F}^1$ satisfy the Axiom A and the no-cycle properties (see also a result by Aoki~\cite{Aoki}). We point out that classic results imply that being in $\mathcal{F}^1$ is a necessary condition to satisfy the Axiom A and the no-cycle condition (see~\cite{Ma1} and the references wherein).

For the continuous-time case the analogous to the set $\mathcal{F}^1$ is traditionally denoted by $\mathcal{G}^1$ and, a flow in it, is called a star flow. Obviously, in this setting, the hyperbolicity of the flow equilibria (singularities of the vector field) is also imposed.

It is well known that the dissipative star flow defined by the Lorenz differential equations (see e.g. ~\cite{WT})  belongs to $\mathcal{G}^1$. However, the hyperbolic saddle-type singularity is accumulated by (hyperbolic) closed orbits and they are contained in the nonwandering set preventing the flow to be Axiom A. Due to the technical difficulties presented in the flow setting, the problem of knowing if every (nonsingular) dissipative star flow satisfies the Axiom A and the no-cycle condition remains unsolved for almost 20 years. This central result was proved by Gan and Wen (\cite{GanWen}).

If we consider flows that are divergence-free and define $\mathcal{G}^1_{div=0}$, which means that the star property is satisfied when one restricts to the conservative setting (but possibly not in the broader space of dissipative flows), using a completely different approach, based in conservative-type seminal ideas of Ma\~n\'e, two of the authors (see~\cite{MBJR}) proved recently that any divergence-free star vector field defined in a closed three-dimensional manifold does not have singularities and moreover it is Anosov (the manifold is uniformly hyperbolic).

In this paper we follow the strategy described in \cite{MBJR}, in order to study the setting of Hamiltonian flows defined on a four-dimensional compact symplectic manifold $(M,\omega)$. For that, we use specific tools and several recent results on conservative three-dimensional flows and on Hamiltonian flows. It is worth pointing out that part of the difficulty of our problem consists in transposing in a proper way concepts from the conservative flow setting to the Hamiltonian one.

To state our main result let  us first recall that a critical point of $H$ is a singularity of the associated vector field. Let $\mathcal{G}^2(M)$ denote the set of Hamiltonian star systems, define in a similar way of the previous ones (we refer the reader to Definition~\ref{hss}), and denote by $\mathcal{E}^\star_{H,e}$ the connected component of the energy level set $H^{-1}(\{e\})$ associated to the star property.
We prove the following.

\begin{maintheorem}\label{mainth}
If $(H,e)\in\mathcal{G}$$^{2}(M)$ and Crit $(H|_{\mathcal{E}^\star_{H,e}})=\emptyset$ then $X_H^t|_{\mathcal{E}^\star_{H,e}}$ is Anosov; if  $\mathcal{E}^\star_{H,e}$ has critical points then there exists $\tilde{e}$, arbitrarily close to $e$, such that $X_H^t|_{\mathcal{E}^\star_{H,\tilde{e}}}$ is Anosov.
\end{maintheorem}

As a consequence of Theorem \ref{mainth} we obtain the following result.

\begin{corollary}\label{maincorollary}
In dimension four, the boundary of the Anosov Hamiltonian level set has no isolated points.
\end{corollary}

We also prove that the Anosov Hamiltonian levels form an open set (Theorem~\ref{Aopen}) and are (strongly) structurally stable (Theorem~\ref{Ass}). Notice that, due to the openness of the Anosov Hamiltonian levels, the reciprocal of Theorem~\ref{mainth} is trivial.
Finally, we show that structurally stable Hamiltonian levels are Anosov (Theorem~\ref{ssA}).

\medskip 
In Section \ref{prel} we present all needed ingredients in the Hamiltonian framework. Section \ref{tools} contains some useful perturbation lemmas and some auxiliary results that will be needed in Section \ref{resultproof}, which contains the proof of the main theorem, obtained in two steps. Given a Hamiltonian star system defined on a regular energy surface we prove that the associated transversal linear Poincar\'{e} flow admits a dominated splitting over the considered energy surface. With this, afterwards we show how we can reach hyperbolicity.  

In Section~\ref{anosov}, following classic arguments of hyperbolic dynamics (see \cite{brin,Katok}), we present the proof of the openness and strong structural stability of Anosov Hamiltonian levels defined on a symplectic $2d$-dimensional manifold $M$. This result is used by several authors and here we present a proof for future use. For this, the result on the continuity of hyperbolic sets will be very useful. We also prove that, in dimension four, structurally stable Hamiltonian levels are Anosov.

\end{section}

%%%%%%%%%%%%%%%%%%%%%%%%%%%%%%%%%%%%%%%%%%%%%%%%%%%%%%%%%%%%%%%%%%%%%%%%%%%%%%%%%%%%%%%%%%%%%%%%%%%%%%%%%%%%%%%%%%%%%%%%%%%%%%%%%%%%%%%%%%%%%%%%%%%
%%%%%%%%%%%%%%%%%%%%%%%%%%%%%%%%%%%%%%%%%%%%%%%%%%%%%%%%%%%%%%%%%%%%%%%%%%%%%%%%%%%%%%%%%%%%%%%%%%%%%%%%%%%%%%%%%%%%%%%%%%%%%%%%%%%%%%%%%%%%%%%%%%%
%%%%%%%%%%%%%%%%%%%%%%%%%%%%%%%%%%%%%%%%%%%%%%%%%%%%%%%%%%%%%%%%%%%%%%%%%%%%%%%%%%%%%%%%%%%%%%%%%%%%%%%%%%%%%%%%%%%%%%%%%%%%%%%%%%%%%%%%%%%%%%%%%%%

\begin{section}{Preliminaries}\label{prel}
%%%%%%%%%%%%%%%%%%%%%%%%%%%%%%%%%%%%%%%%%%%%%%%%%%%%%%%%%%%%%%%%%%%%%%%%%%%%%%%%%%%%%%%%%%%%%%%%%%%%%%%%%%%%%%%%%%%%%%%%%%%%%%%%%%%%%%%%%%%%%%%%%%%

\begin{subsection}{Notation and basic definitions}\label{notation}

Let $(M,\omega)$ be a compact symplectic manifold, where $M$ is a four-dimensional, smooth and compact manifold endowed with a symplectic structure $\omega$, i.e. a skew-symmetric and nondegenerate 2-form on the tangent bundle $TM$.

A $C^1$ diffeomorphism $f$ defined on $(M,\omega)$ is called a \textit{symplectomorphism} if, taking $p\in M$, $\omega(u,v)=\omega(D_pf(u),D_pf(v)),\; \forall\: (u,v)\in T_pM\times T_pM$.
The set of all symplectomorphism forms a group under composition, called the \textit{symplectic group}, denoted by $Sp(M,\omega)$. The condition that $f\in Sp(M,\omega)$ can be expressed in matrix notation. Since $\omega$ is a symplectic form, there is an ordered basis of $M$ such that the matrix of $\omega$ is
\begin{equation}
J=\left[
\begin{array}{ll}
	0 & I_2\\
	-I_2 & 0
\end{array}\nonumber
\right],
\end{equation}
where $I_2$ denotes the identity matrix with dimension $2$, once $M$ has dimension $4$. Note that $J^{-1}=J^{\:T}=-J$ and $J^{\:2}=-I_2$. Take $f\in Diff(M, \omega)$ such that, relatively with the mentioned ordered basis, $D_pf$ has matrix 
\begin{equation}
A=\left[
\begin{array}{ll}
	A_1 & A_2\\
	A_3 & A_4
\end{array}\nonumber
\right],
\end{equation}
where $A_1, A_2, A_3, A_4$ are $2\times 2$ matrices. 

We have that $f\in Sp(M,\omega)$ if and only if $A^TJ\:A=J$ or, equivalently, $A_1^TA_3$ and $A_2^TA_4$ are symmetric and $A_1^TA_4-A_3^TA_2=I_2$.

The next elementary result states some conclusions about the eigenvalues of the linear part of a symplectomorphism.

\begin{theorem} (Symplectic eigenvalue theorem, \cite{Abraham})\label{sympeigen}
Let $f\in Sp(M,\omega)$, $p\in M$ and $\lambda$ an eigenvalue of $D_pf$ of multiplicity $k$. Then $1/\lambda$ is an eigenvalue of $D_pf$ of multiplicity $k$. Moreover, the multiplicities of the eigenvalues $+1$ and $-1$, if they occur, are even.
\end{theorem}

We will be interested in the Hamiltonian dynamics of real-valued $C^s$ functions on $M$, $2\leq s\leq \infty$, constant on each connected component of the boundary of $M$, called Hamiltonians, whose set we denote by $C^s(M,\mathbb{R})$.  For any \textit{Hamiltonian} function $H:M\longrightarrow \mathbb{R}$ there is a co\-rres\-pon\-ding \textit{Hamiltonian vector field} $X_H:M\longrightarrow TM$,  tangent to the boundary of $M$, and determined by the condition 
\begin{displaymath}
d_pH(u)=\omega(X_H(p),u), \;\forall u\in T_pM, 
\end{displaymath}
where $p\in M$. In the matricial framework, this is equivalent to  $X_H(p)=J \cdot d_pH$. So, notice that $D_pX_H=J\cdot D^2_pH$. The nondegeneracy of the form $\omega$ guarantees that $X_H$ is well defined, while the skew-symmetry of $\omega$ leads to conservative properties for the Hamiltonian vector field. Notice that $H$ is $C^s$ if and only if $X_H$ is $C^{s-1}$ (see section~\ref{measure}). Here we consider the space of the Hamiltonian vector fields endowed with the $C^1$ topology, and for that we consider $C^s(M,\mathbb{R})$ equipped with the $C^2$ topology. This space can also be endowed with the $C^r$-topology, $1 \leq r<s$, and we denote by  $\|H-G\|_r$ the $C^r$-distance between $H$ and $G$.

The Hamiltonian vector field $X_H$ generates the \textit{Hamiltonian flow} $X_H^t$, a smooth 1-parameter group of symplectomorphisms on $M$ sa\-tis\-fying $\frac{d}{dt}X_H^t=X_H (X_H^t)$ and $X_H^0=id$. We also consider the \textit{tangent flow} $D_pX_H^t:T_pM\longrightarrow T_{X_H^t(p)}M$, for $p\in M$, that satisfies the linearized differential equation $\frac{d}{dt}D_pX_H^t=(D_{X_H^t(p)}X_H)\circ D_pX_H^t$, where $D_pX_H:T_pM\longrightarrow T_pM$.

Once $\omega$ is non-degenerate, given $p \in M$, $d_pH=0$ is equivalent to $X_H(p)=0$, and in this case we say that $p$ is a \textit{critical point} of $H$ or  a \textit{singularity} of $X_H$. A point is said to be \textit{regular} if it is not a critical point. We denote by $\mathcal{R}$ the set of regular points of $H$, by Crit ($H$) the set of critical points of $H$ and by Sing ($X_H$) the set of singularities of $X_H$. Taking in account the relation between $H$ and $X_H$, observe that $\text{Sing} (X_H)=\text{Crit} (H)$. 

A closed orbit of $H$ with period $\pi$ is a closed orbit of  $X_H^t$ with period $\pi$. Given a regular point $x$ of a Hamiltonian $H$, we define the arc $X_H^{[t_1,t_2]}(x)=\{X_H^t(x), \, t \in [t_1,t_2]\}$; given a transversal section $\Sigma$ of $x$, a flowbox associated to $\Sigma$ is defined by $\mathcal{F}(x)=X_H^{[-\tau_1,\tau_2]}(\Sigma)$, where $\tau_1, \tau_2$ are chosen small such that $\mathcal{F}(x)$ is a neighbourhood of $x$ foliated by regular orbits.

Let $H$ be a Hamiltonian. Any scalar $e\in H(M)\subset \mathbb{R}$ is called an \textit{energy} of $H$ and $H^{-1}(\left\{e\right\})=\left\{p\in M: H(p)=e\right\}$ is the corresponding \textit{energy level} set. It is $X_H^t$-invariant. An \textit{energy surface} $\mathcal{E}_{H,e}$ is a connected component of $H^{-1}(\left\{e\right\})$; we say that it is regular if it does not contain critical points and in this case $\mathcal{E}_{H,e}$ is a regular compact 3-manifold. Moreover, $H$ is constant on each connected component $\mathcal{E}_{H,e}$ of the boundary $\partial M$.

Due to the compactness of $M$, given a Hamiltonian function $H$ and $e\in H(M)$ the energy level $H^{-1}(\left\{e\right\})$ is the union of a finite number of disjoint compact connected components, separated by a positive distance.  Given $e \in H(M)$,  the pair $(H,e) \subset C^2(M,\mathbb{R}) \times \mathbb{R}$ is called a \textit{Hamiltonian level}; if we fix $\mathcal{E}_{H,e}$ and a small neighbourhood $\mathcal{W}$ of $\mathcal{E}_{H,e}$ there exist a small neighbourhood $\mathcal{U}$ of $H$ and $\delta>0$ such that for all $\tilde{H} \in \mathcal{U}$ and $\tilde{e} \in ]e-\delta,e+\delta[$ one has that $\tilde{H}^{-1}(\{\tilde{e}\})\cap \mathcal{W}=\mathcal{E}_{\tilde{H},\tilde{e}}$. We call $\mathcal{E}_{\tilde{H},\tilde{e}}$ the analytic continuation of $\mathcal{E}_{H,e}$.

On $M$ we also fix a Riemannian structure which induces a norm $\left\|.\right\|$ on the fibers $T_pM$, $\forall \:p\in M$. We will use the standard norm of a bounded linear map $L$ given by %$\displaystyle\left\|L\right\|=\sup_{\left\|u\right\|=1}\left\|L(u)\right\|$.
\begin{equation}
\displaystyle\left\|L\right\|=\sup_{\left\|u\right\|=1}\left\|L(u)\right\|.\nonumber
\end{equation}

A metric on $M$ can be derived in the usual way through the Darboux's charts and it will be denoted by $dist$. Hence, we define the open balls $B_r(p)$ of the points $x\in M$ verifying $dist(x,p)<r$.

We end this section introducing a crucial definition. We introduce the notion  of {\it Hamiltonian star system} which is similar to the one of star conservative flow.

\begin{definition}\label{hss}

A Hamiltonian level $(H,e)$  is a \textit{Hamiltonian star system}
if there exist a $C^2$-neighbourhood $\mathcal{U}$ of $H$ and $\delta>0$ such that if $\tilde{H}\in \mathcal{U}$ and $\tilde{e}\in (e-\delta,e+\delta)$, then all the closed orbits and all the critical points of $\tilde{H}$ on $\mathcal{E}^\star_{\tilde{H},\tilde{e}}$ are hyperbolic, where $\mathcal{E}^\star_{\tilde{H},\tilde{e}}$ is the analytic continuation of $\mathcal{E}^\star_{H,e}$. 

We denote by $\mathcal{G}^{2}(M) \subset C^2(M,\mathbb{R}) \times \mathbb{R}$  the set of all Hamiltonian star systems.
\end{definition}

\end{subsection}
\begin{subsection}{Measure and topological dimension}\label{measure}

The symplectic manifold $(M,\omega)$ is also a volume manifold by \textit{Liouville's Theorem} (see for example~\cite{Abraham}). So, the volume form $\omega^2=\omega\wedge\omega$ induces a measure $\mu$ on $M$ that is the Lebesgue measure associated to $\omega^2$.

Notice that the measure $\mu$ on $M$ is preserved by the Hamiltonian flow. So, given any energy $e$ of a Hamiltonian $H$, on each regular energy surface $\mathcal{E}_{H,e}\subset H^{-1}(\left\{e\right\})\subset M$ we induce a volume form $\omega_{\mathcal{E}_{H,e}}$:
\begin{align}
\omega_{\mathcal{E}_{H,e}}:\;&T_p\mathcal{E}_{H,e} \times T_p\mathcal{E}_{H,e}\times T_p\mathcal{E}_{H,e} \longrightarrow \mathbb{R}\nonumber\\
& (u,v,w)\longmapsto \omega^2(d_pH,u,v,w), \;\: \forall \:p\in \mathcal{E}_{H,e}\nonumber.
\end{align}
We have that $\omega_{\mathcal{E}_{H,e}}$ is $X_H^t$-invariant. So, it induces an invariant volume measure $\mu_{\mathcal{E}_{H,e}}$ on $\mathcal{E}_{H,e}$ that is finite, since energy surfaces are compact. Notice that, under these conditions, we can apply the \textit{Poincar\'{e} Recurrence Theorem}. Therefore, we have that $\mu_{\mathcal{E}_{H,e}}$-a.e. $x\in \mathcal{E}_{H,e}$ is recurrent.

\begin{definition}
We say that the measure $\mu_{\mathcal{E}_{H,e}}$ is \textit{ergodic} if, for any $X_H^t$-invariant subset of $\mathcal{E}_{H,e}$, say $\Lambda$, we have that $\mu_{\mathcal{E}_{H,e}}(\Lambda)=0$ or $\mu_{\mathcal{E}_{H,e}}(\Lambda)=1$.
\end{definition}

There are different definitions of the \textit{topological dimension} of a topological space $X$, say $\dim(X)$, which are equivalent just for separable metrizable spaces. In the formulation of Menger, the dimension of a space is the least integer $n$ for which every point has arbitrarily small neighbourhoods whose boundaries have dimension less than $n$. There is a result, due to Szpilrajn (\cite{szpil}), relating the topological dimension with the Lebesgue measure. 

\begin{definition}
Let $n \geq 0$. 
We say that $X$ has dimension $\leq n$, $\dim (X)\leq n$, if there exists a basis of $X$ made up of open sets whose boundaries have dimension $\leq n-1$. Also, we say that $X$ has dimension $n$ if $\dim (X)\leq n$ is true and $\dim (X)\leq n-1$ is false.
\end{definition}
This property is topologically invariant. Even more, if $X$ is compact we have that $\dim (X)\leq n$ if and only if any two distinct points (or disjoint closed sets) can be separated by a closed set of dimension $\leq n-1$. 

The following result relates a metrical concept with a topological one.

\begin{theorem}\label{szpil} (E. Szpilrajn, \cite{szpil}) Let $X\subset\mathbb{R}^n$ be a topological space. If $X$ has zero Lebesgue measure then $\dim(X)<n$.
\end{theorem}
\medskip%p.104

\end{subsection}

%%%%%%%%%%%%%%%%%%%%%%%%%%%%%%%%%%%%%%%%%%%%%%%%%%%%%%%%%%%%%%%%%%%%%%%%%%%%%%%%%%%%%%%%%%%%%%%%%%%%%%%%%%%%%%%%%%%%%%%%%%%%%%%%%%%%%%%%%%%%%%%%%%%

\begin{subsection}{Lyapunov exponents}\label{Lyap}

Take $H \in C^2(M,\mathbb{R})$. Since $DX_H^t$ is measure preserving, we have a version of \textit{Oseledets' Theorem} (\cite{Oseledets}) for four dimensional Hamiltonians. For $\mu$-a.e. point $x\in M$ we have two possible splittings:

\begin{enumerate}
	\item $T_xM=E_x$, with $E_x$ four-dimensional and
\begin{displaymath}
		\displaystyle \lim_{t\rightarrow \pm\infty}\frac{1}{t}\log \left\|DX_H^t(x)\:v\right\|=0, \ \ \forall v\in E_x\setminus\{0\}, or 
\end{displaymath}

\item $T_xM=\mathbb{R}X_H(x)\oplus E_x^0\oplus E_x^+\oplus E_x^-$, with each one of these subspaces being one-dimensional and $DX_H^t$-invariant, and

\begin{itemize}
	\item  $\displaystyle \lim_{t\rightarrow \pm\infty}\frac{1}{t}\log \left\|DX_H^t(x)\:v\right\|=0, \ \ \forall v\in \mathbb{R}X_H(x)\oplus E_x^0\setminus\{0\}$ ;
	\item $\lambda^+(H,x):=\displaystyle \lim_{t\rightarrow \pm\infty}\frac{1}{t}\log \left\|DX_H^t(x)\:v\right\|>0, \ \ \forall v\in E_x^+\setminus\{0\}$;
	\item $\lambda^-(H,x):=\displaystyle \lim_{t\rightarrow \pm\infty}\frac{1}{t}\log \left\|DX_H^t(x)\:v\right\|=-\lambda^+(H,x), \ \ \forall v\in E_x^-\setminus\{0\}$.
		
\end{itemize}
\end{enumerate}

The splitting of the tangent bundle is called \textit{Oseledets' splitting} and the real numbers $\lambda^{\pm}(H,x)$ are called the \textit{Lyapunov exponents}. The full $\mu$-measure set of the \textit{Oseledets points} is denoted by $\mathcal{O}(H)$.

\end{subsection}

%%%%%%%%%%%%%%%%%%%%%%%%%%%%%%%%%%%%%%%%%%%%%%%%%%%%%%%%%%%%%%%%%%%%%%%%%%%%%%%%%%%%%%%%%%%%%%%%%%%%%%%%%%%%%%%%%%%%%%%%%%%%%%%%%%%%%%%%%%%%%%%%%%%

\begin{subsection}{Transversal linear Poincar\'{e} flow of a Hamiltonian}\label{Poin}

Let $H\in C^2(M,\mathbb{R})$, $e\in H(M)$, $x\in \mathcal{R}$ and take the orthogonal splitting $T_xM=N_x\oplus \mathbb{R}X_H(x)$, where $N_x=(\mathbb{R}X_H(x))^{\bot}$ is the normal fiber at $x$ and $\mathbb{R}X_H(x)$ denotes the vector field direction. Consider the skew-product automorphism of vector bundles
\begin{align}
DX_H^t:& \ \ T_{\mathcal{R}}M\longrightarrow T_{\mathcal{R}}M\nonumber\\
& \ \ (x,v)\longmapsto (X_H^t(x), DX_H^t(x)v).\nonumber
\end{align}

Since, in general, the subbundle $N_{\mathcal{R}}$ is not $DX_H^t$-invariant, we are going to relate the $DX_H^t$-invariant quotient space $\tilde{N_{\mathcal{R}}}=T_{\mathcal{R}}M /\mathbb{R}X_H$ with an isometric isomorphism $h_1:N_{\mathcal{R}}\rightarrow \tilde{N}_{\mathcal{R}}$. Denote the canonical orthogonal projection by $\Pi_{{\mathcal{R}}}: T_{\mathcal{R}}M\rightarrow N_{\mathcal{R}}$. So, the unique map
\begin{align}
P_H^t: &\ \ N_{\mathcal{R}}\rightarrow N_{\mathcal{R}}\nonumber\\
&\ \ (x,v)\mapsto \Pi_{X_H^t(x)}\circ DX_H^t(x)v\nonumber
\end{align}
such that $h_1\circ P_H^t=DX_H^t\circ h_1$ is called the \textit{linear Poincar\'{e} flow} associated to $H$, which was first introduced by Doering in~\cite{Doering}.

Now consider
\begin{displaymath}
\mathcal{N}_x=N_x\cap T_xH^{-1}(\left\{e\right\}) 
\end{displaymath}
where $T_xH^{-1}(\left\{e\right\})=Ker\: dH(x)$ is the tangent space to the energy level set with $e=H(x)$. Thus, $\mathcal{N_R}$ is $P_H^t$-invariant and we can define the \textit{transversal linear Poincar\'{e} flow} for $H$
\begin{align}
\Phi_H^t: &\ \ \mathcal{N}_{\mathcal{R}}\rightarrow \mathcal{N}_{\mathcal{R}}\nonumber\\
&\ \ (x,v)\mapsto \Pi_{X_H^t(x)}\circ DX_H^t(x)\:v\nonumber
\end{align}
that is $\Phi_H^t=P_H^t|_{\mathcal{N}_{\mathcal{R}}}$, it is a linear symplectomorphism for the symplectic form induced by $\omega$ on $\mathcal{E}_{H,e}$.

If $x\in \mathcal{R}\cap \mathcal{O}(H)$ and $\lambda^+(H,x)>0$, the Oseledets splitting on $T_xM$ induces a $\Phi_H^t(x)$-invariant splitting $\mathcal{N}_x=\mathcal{N}_x^+\oplus \mathcal{N}_x^-$, where $\mathcal{N}_x^{\pm}=\Pi_x(E_x^{\pm})$ are one dimensional subbundles. It is straightforward to see that the Lyapunov exponents of this splitting coincide with that ones of the $DX_H^t$-invariant splitting (see \cite[Lemma 2.1]{MBJLD}). 

\bigskip
Let $\Gamma\subset M$ be a closed orbit of period $\pi$. The characteristic multipliers of $\Gamma$ are the eigenvalues of $\Phi_H^{\pi}(p)$, which are independent of the point $p\in\Gamma$. If $\chi$ is a characteristic multiplier of a closed orbit $\Gamma$ of period $\pi$, then the associated Lyapunov exponent is $\lambda=\log(\chi)/\pi$. In our context the product of the characteristic multipliers is equal to one, or equivalently the sum of the two Lyapunov exponents is equal to zero (cf. Theorem~\ref{sympeigen} ). We say that $\Gamma$ is
\begin{itemize}	
	\item \textit{hyperbolic} when the characteristic multipliers have modulus different from 1;
	\item \textit{parabolic} when the characteristic multipliers are real and of modulus 1; 
  \item \textit{elliptic} when the two characteristic multipliers are simple, non-real and of mo\-du\-lus 1.
\end{itemize}

So, under small perturbations, hyperbolic and elliptic orbits are stable, whilst parabolic ones are unstable. 

\end{subsection}

%%%%%%%%%%%%%%%%%%%%%%%%%%%%%%%%%%%%%%%%%%%%%%%%%%%%%%%%%%%%%%%%%%%%%%%%%%%%%%%%%%%%%%%%%%%%%%%%%%%%%%%%%%%%%%%%%%%%%%%%%%%%%%%%%%%%%%%%%%%%%%%%%%%

\begin{subsection}{Anosov Hamiltonian  level}\label{hyper}

Let $H\in C^2(M,\mathbb{R})$. Given any compact and $X_H^t$-invariant set $\Lambda\subset \mathcal{E}_{H,e}$, we say that $\Lambda$ is a hyperbolic set for $X_H^t$ if there exists $m\in \mathbb{N}$, a constant $\theta\in(0,1)$ and a $DX_H^t$-invariant splitting $T_{\Lambda}\mathcal{E}_{H,e}=E^-\oplus  E \oplus E^+ $ such that, for all $x\in \Lambda$, we have:
\begin{itemize}
	\item $\left\|DX_H^m(x)|_{E^-_x}\right\|\leq \theta \ \ $ (uniform contraction),
	\item $\left\|DX_H^{-m}(x)|_{E^+_x}\right\|\leq\theta \ \ $ (uniform expansion),
	\item $E=E^0\oplus \mathbb{R}X_H(x)$, and includes the direction of the gradient of $H$.
\end{itemize}

\begin{definition}
We say that a Hamiltonian level $(H,e)\in C^2(M,\mathbb{R})\times H(M)$ is Anosov if and only if there exists an energy surface $\mathcal{E}_{H,e}$ which is hyperbolic for $X_H^t$. For $d\in\mathbb{N}\setminus\{1\}$, let $\mathcal{A}^{2d}(M) \subset C^2(M, \mathbb{R})\times  \mathbb{R}$ denote the set of Anosov Hamiltonian levels defined on a $2d$-dimensional manifold $M$.
\end{definition}

We observe that if $(H,e)$ is Anosov then the energy surface $\mathcal{E}_{H,e}$ is regular, that is does not contain critical points of $H$.

%%%%%%%%%%%%%%%%%%%%%%%%%%%%%%%%%%%%%%%%%%%%%%%%%%%%%%%%%%%%%%%%%%%%%%%%%%%%%%%%%%%%%%%%%%%%%%%%%%%%%%%%%%%%%%%%%%%%%%%%%%%%%%%%%%%%%%%%%%%%%%%%%%%

\begin{subsubsection}{Relation with $\Phi_H^t$}

Similarly, we can define a hyperbolic structure for the transversal linear Poincar\'{e} flow $\Phi_H^t$. A compact and $X_H^t$-invariant set $\Lambda\subset \mathcal{E}_{H,e}$ is called hyperbolic if there exists $m\in \mathbb{N}$ and a constant $\theta\in(0,1)$ such that, for every $x\in \Lambda$, $\mathcal{N}_x=\mathcal{N}^-_x\oplus \mathcal{N}^+_x$ and

\begin{itemize}
	\item $\left\|\Phi_H^m(x)|_{\mathcal{N}^-_x}\right\|\leq \theta$,
	\item $\left\|\Phi_H^{-m}(x)|_{\mathcal{N}^+_x}\right\|\leq \theta$,
	\item $\Phi_H^t(x)\mathcal{N}^-_x=\mathcal{N}^-_{X_H^t(x)}\:$ and $\:\Phi_H^t(x)\mathcal{N}^+_x=\mathcal{N}^+_{X_H^t(x)}$, $\forall \: t\geq 0$.
\end{itemize}

If $\Phi_H^t|{\Lambda}$ is a hyperbolic vector bundle automorphism, we say that $\Lambda$ is hyperbolic for $\Phi_H^t$ on $\Lambda$. Notice that, by compactness of $\Lambda$, to ensure that $\Lambda$ is hyperbolic it is enough to show hyperbolicity for just one $m\in \mathbb{N}$.

\bigskip
Following the ideas due to Hirsch, Pugh and Shub (\cite{HPS}), given a hyperbolic set $\Lambda$, we analogously have that $\mathcal{N}_x^-$ and $\mathcal{N}_x^+$ depend continuously on $x\in\Lambda$.

Next lemma relates the hyperbolicity for $\Phi_H^t$ with the hyperbolicity for $X_H^t$ and it is an extension to the Hamiltonian setting of a result due to Doering \cite[Proposition 1.1]{Doering} and is supported on an abstract invariant manifold theory result of Hirsch, Pugh and Shub (\cite[Lemma 2.18]{HPS}).

\begin{lemma}\label{hyperbolicpoinc}
Let $\Lambda$ be a $X_H^t$-invariant, regular and compact set. Then $\Lambda$ is hyperbolic for $X_H^t$ if and only if the induced transversal linear Poincar\'{e} flow $\Phi_H^t$ is hyperbolic on $\Lambda$.
\end{lemma}

\end{subsubsection}

%%%%%%%%%%%%%%%%%%%%%%%%%%%%%%%%%%%%%%%%%%%%%%%%%%%%%%%%%%%%%%%%%%%%%%%%%%%%%%%%%%%%%%%%%%%%%%%%%%%%%%%%%%%%%%%%%%%%%%%%%%%%%%%%%%%%%%%%%%%%%%%%%%%
%%%%%%%%%%%%%%%%%%%%%%%%%%%%%%%%%%%%%%%%%%%%%%%%%%%%%%%%%%%%%%%%%%%%%%%%%%%%%%%%%%%%%%%%%%%%%%%%%%%%%%%%%%%%%%%%%%%%%%%%%%%%%%%%%%%%%%%%%%%%%%%%%%%

\end{subsection}

%%%%%%%%%%%%%%%%%%%%%%%%%%%%%%%%%%%%%%%%%%%%%%%%%%%%%%%%%%%%%%%%%%%%%%%%%%%%%%%%%%%%%%%%%%%%%%%%%%%%%%%%%%%%%%%%%%%%%%%%%%%%%%%%%%%%%%%%%%%%%%%%%%%
%%%%%%%%%%%%%%%%%%%%%%%%%%%%%%%%%%%%%%%%%%%%%%%%%%%%%%%%%%%%%%%%%%%%%%%%%%%%%%%%%%%%%%%%%%%%%%%%%%%%%%%%%%%%%%%%%%%%%%%%%%%%%%%%%%%%%%%%%%%%%%%%%%%

\begin{subsection}{Dominated splitting}\label{dom}

Dominated splitting is a weaker form of hyperbolicity.
 
\begin{definition}
Let $H\in C^2(M,\mathbb{R})$ and $\Lambda\subset \mathcal{R}$ be a compact and $X_H^t$-invariant set and $m\in \mathbb{N}$. A splitting of the bundle ${\mathcal{N}}_{\Lambda}={\mathcal{N}}_{\Lambda}^-\oplus{\mathcal{N}}_{\Lambda}^+$ is an $m$-\textit{dominated splitting} for the transversal linear Poincar\'{e} flow if it is continuous,  $\Phi_H^t$-invariant and there is a constant $\theta\in(0,1)$ such that
\begin{displaymath}
\dfrac{\left\|\Phi_H^m(x)|_{{\mathcal{N}}_x^-}\right\|}{\left\|\Phi_H^m(x)|_{{\mathcal{N}}_{x}^+}\right\|}\leq\theta, \: \: x\in \Lambda.
\end{displaymath}
We call ${\mathcal{N}}_{\Lambda}={\mathcal{N}}_{\Lambda}^-\oplus{\mathcal{N}}_{\Lambda}^+$ a dominated splitting if it is $m$-dominated for some $m\in \mathbb{N}$.
\end{definition}

Let us now present some useful properties of a dominated splitting on $\Lambda$. For more details see \cite{BDV}.
\begin{itemize}
	\item \textit{Uniqueness}: the dominated splitting is unique if one fixes the dimension of the subbundles. So, due to our low dimensional assumption and to the Symplectic Eigenvalue Theorem (Theorem \ref{sympeigen}), the decomposition is unique.

\item\textit{Continuity}: every dominated splitting is continuous, i.e. the subbundles $\mathcal{N}_x^-$ and $\mathcal{N}_x^+$ depend continuously on the point $x$.

\item\textit{Transversality}: the angles between $\mathcal{N}^-$ and $\mathcal{N}^+$ are bounded away from zero on $\Lambda$.

\end{itemize}

Ahead, Lemma \ref{lemasing} will show how we can reach, under some conditions, hyperbolicity from the dominated splitting.

\end{subsection}

%%%%%%%%%%%%%%%%%%%%%%%%%%%%%%%%%%%%%%%%%%%%%%%%%%%%%%%%%%%%%%%%%%%%%%%%%%%%%%%%%%%%%%%%%%%%%%%%%%%%%%%%%%%%%%%%%%%%%%%%%%%%%%%%%%%%%%%%%%%%%%%%%%%

\end{section}

%%%%%%%%%%%%%%%%%%%%%%%%%%%%%%%%%%%%%%%%%%%%%%%%%%%%%%%%%%%%%%%%%%%%%%%%%%%%%%%%%%%%%%%%%%%%%%%%%%%%%%%%%%%%%%%%%%%%%%%%%%%%%%%%%%%%%%%%%%%%%%%%%%%

\begin{section}{Some main tools}\label{tools}

\begin{subsection}{Perturbation lemmas}

Next lemma is a version of the Clo\-sing Lemma, which can be easily obtained combining  Arnaud's Clo\-sing Lemma (\cite{Ar}) with Pugh and Robinson's Closing Lemma for Hamiltonians (\cite{PughRob}). It states that the orbit of a non-wandering point can be approximated for a very long time by a closed orbit of a nearby Hamiltonian.

\begin{lemma} \label{closing} Take $H_1 \in C^2(M,\mathbb{R})$, a non-wandering point $x\in M$ and $\epsilon$, $r$, $\tau >0$. Then we can find $H_2\in C^2(M,\mathbb{R})$, a closed orbit $\Gamma$ of $H_2$ with  period $\pi$, $p\in \Gamma$ and a map $g:[0,\tau]\rightarrow [0,\pi]$ close to the identity such that:
\begin{itemize}
	\item $H_2$ is $\epsilon$-$C^2$-close to $H_1$,
	\item $dist\Bigl(X^t_{H_1}(x), X^{g(t)}_{H_2}(p)\Bigr)<r$, $0\leq t\leq \tau$,
	\item $H_1=H_2$ on $M\backslash A$ where $A= \bigcup_{0\leq t\leq \tau}\Bigl( B_r\big(X^t_{H_1}(p)\big)\Bigr)$.
\end{itemize}
\end{lemma}

Next we present a result of Vivier which is a version of Franks' lemma for Hamiltonians (see~\cite{V2}). Roughly, it says that we can realize a Hamiltonian corresponding to a given perturbation of the transversal linear Poincar\'{e} flow.

\begin{lemma}\label{Frank}  Take $H_1 \in C^2(M,\mathbb{R})$, $\epsilon$, $\tau >0$ and $x\in M$. Then, there exists $\delta>0$ such that for any flowbox $V$ of an injective arc of orbit $X_{H_1}^{[0,t]}(x)$, $t\geq \tau$, and a transversal symplectic $\delta$-perturbation $F$ of $\Phi_{{H_1}}^t(x)$, there is $H_2\in C^2(M,\mathbb{R})$ satisfying:
\begin{itemize}
	\item $H_2$ is $\epsilon$-$C^2$-close to $H_1$,
	\item $\Phi_{{H_2}}^t(x)=F$,
	\item $H_1=H_2$ on $X_{H_1}^{[0,t]}(x)\cup (M\backslash V)$.
\end{itemize}
\end{lemma}

\end{subsection}

%%%%%%%%%%%%%%%%%%%%%%%%%%%%%%%%%%%%%%%%%%%%%%%%%%%%%%%%%%%%%%%%%%%%%%%%%%%%%%%%%%%%%%%%%%%%%%%%%%%%%%%%%%%%%%%%%%%%%%%%%%%%%%%%%%%%%%%%%%%%%%%%%%%

\begin{subsection}{Auxiliary results}

This section presents several useful results that are going to be applied in the proof of the main theorem. 

In the presence of a weakly hyperbolic periodic orbit, the next two lemmas give us conditions  to create a nearby elliptic closed orbit via a small perturbation. These results can be easily obtained by combining the techniques developed in Lemma~\ref{Frank} with the arguments in ~\cite{MBPD}. 

\begin{lemma} (\cite[Proposition 3.8]{MBPD})\label{prop1}
Let $H\in C^s(M,\mathbb{R})$, $2\leq s\leq \infty$, and $\epsilon>0$. There is $\theta>0$ such that for any closed hyperbolic orbit $\Gamma$ with  period $\tau>1$ and angle between $\mathcal{N}_q^+$ and $\mathcal{N}_q^-$ smaller than $\theta$, for $q\in \Gamma$, there is $\tilde{H}\in C^{\infty}(M,\mathbb{R})$ $\epsilon$-$C^2$-closed to $H$ for which $\Gamma$ is elliptic with  period $\tau$.
\end{lemma}

\begin{lemma} (\cite[Proposition 3.10]{MBPD})\label{prop2}
Let $H\in C^s(M,\mathbb{R})$, $2\leq s\leq \infty$, and $\epsilon, \theta>0$. There exist $m,T\in\mathbb{N}$ $(T>>m)$ such that, if a hyperbolic closed orbit $\Gamma$ with  period $\tau>T$ satisfies:
\begin{itemize}
	\item angle between $\mathcal{N}_q^+$ and $\mathcal{N}_q^-$ is grater or equal than $\theta$ for all $q\in\Gamma$,
	\item $\Gamma$ has no $m$-dominated splitting,
\end{itemize}
then there exists $\tilde{H}\in C^{\infty}(M,\mathbb{R})$ $\epsilon$-$C^2$-closed to $H$ for which $\Gamma$ is elliptic with  period $\tau$.
\end{lemma}

Conversely, the absence of elliptic periodic orbits for all nearby perturbations implies uniform bounds on hyperbolic orbits with big enough period. This is an easy consequence of the two previous lemmas.

\begin{lemma}\label{lemateta}
Let $H\in C^s(M,\mathbb{R})$, $2\leq s\leq \infty$, and $\epsilon>0$. Set $\theta=\theta(\epsilon,H)$, $m=m(\epsilon, \theta)$ and $T=T(m)$ given by Lemmas \ref{prop1} and \ref{prop2}. Assume that every Hamiltonian $\tilde{H}$ which is $\epsilon$-$C^2$-close to $H$ do not admit elliptic closed orbits. Then, for every such $\tilde{H}$, all closed orbits with period larger that $T$ are hyperbolic, $m$-dominated and with angle between its stable and unstable directions bounded from bellow by $\theta$.
\end{lemma}

\bigskip
In order to close the section, we present a result which is important because we are going to appeal to the techniques involved in its proof.

\begin{theorem} (\cite[Theorem 2]{MBJLD})\label{bessa2} There exists a $C^2$-dense subset $\mathcal{D}$ of $C^2(M,\mathbb{R})$ such that, if $H\in\mathcal{D}$, there exists an invariant decomposition $M=D\cup Z$, mod $0$, satisfying:
\begin{itemize}
	\item $D=\cup_{n\in\mathbb{N}}D_{m_n}$, where $D_{m_n}$ is a set with $m_n$-dominated splitting for $\Phi^t_{H}$,
	\item $X^t_H$ has zero Lyapunov exponents for $p\in Z$.
\end{itemize}
\end{theorem}

%Nota: as sing hiperbolicas sao isoladas. logo sao em numero finito.

\end{subsection}

%%%%%%%%%%%%%%%%%%%%%%%%%%%%%%%%%%%%%%%%%%%%%%%%%%%%%%%%%%%%%%%%%%%%%%%%%%%%%%%%%%%%%%%%%%%%%%%%%%%%%%%%%%%%%%%%%%%%%%%%%%%%%%%%%%%%%%%%%%%%%%%%%%%

\end{section}

%%%%%%%%%%%%%%%%%%%%%%%%%%%%%%%%%%%%%%%%%%%%%%%%%%%%%%%%%%%%%%%%%%%%%%%%%%%%%%%%%%%%%%%%%%%%%%%%%%%%%%%%%%%%%%%%%%%%%%%%%%%%%%%%%%%%%%%%%%%%%%%%%%%
%%%%%%%%%%%%%%%%%%%%%%%%%%%%%%%%%%%%%%%%%%%%%%%%%%%%%%%%%%%%%%%%%%%%%%%%%%%%%%%%%%%%%%%%%%%%%%%%%%%%%%%%%%%%%%%%%%%%%%%%%%%%%%%%%%%%%%%%%%%%%%%%%%%
%%%%%%%%%%%%%%%%%%%%%%%%%%%%%%%%%%%%%%%%%%%%%%%%%%%%%%%%%%%%%%%%%%%%%%%%%%%%%%%%%%%%%%%%%%%%%%%%%%%%%%%%%%%%%%%%%%%%%%%%%%%%%%%%%%%%%%%%%%%%%%%%%%%

\begin{section}{Proof of the results}\label{resultproof}
%%%%%%%%%%%%%%%%%%%%%%%%%%%%%%%%%%%%%%%%%%%%%%%%%%%%%%%%%%%%%%%%%%%%%%%%%%%%%%%%%%%%%%%%%%%%%%%%%%%%%%%%%%%%%%%%%%%%%%%%%%%%%%%%%%%%%%%%%%%%%%%%%%%

\begin{subsection}{Auxiliary lemmas}
			
In this section we split the proof of Theorem~\ref{mainth} into three lemmas. The first lemma deals with conditions that ensure the existence of a dominated splitting on a given energy surface. The second lemma shows how we can  derive hyperbolicity from the existence of a dominated splitting and the third lemma deals with the case of a non-regular  surface energy, which concludes the proof of Theorem~\ref{mainth}.

\begin{lemma}\label{lemasing}
If $(H,e)\in\mathcal{G}$$^{2}(M)$ and $\mathcal{E}^\star_{H,e}$ is regular, then $\Phi_{H}^t$ admits a dominated splitting over $\mathcal{E}^\star_{H,e}$.
\end{lemma}

\begin{proof} 
%Take $(H,e)$ on $\mathcal{G}$$^{2}(M)$, such that every critical point of $H$ on $\mathcal{E}^\star_{H,e}$ is linear hyperbolic, and let $\mathcal{U}\times (e-\delta, e+\delta)$ be a neighbourhood of $(H,e)$ on $\mathcal{G}$$^{2}(M)$, given by the definition of Hamiltonian star system.

%Now let us separate the proof into two cases. First let us assume  that $\mu_{\mathcal{E}^\star_{H,e}}$-a.e. $x\in \mathcal{E}^\star_{H,e}\backslash$Crit$(H|_{\mathcal{E}^\star_{H,e}})$ admits a dominated splitting for $\Phi_H^t$. Notice that the Hamiltonian vector field $X_{H}$ restricted to $\mathcal{E}^\star_{H,e}$ must be modified volume preserving, say there exists a density $\rho(x)$ equivalent to Lebesgue measure which is $X_H^t$-invariant. Since every critical point of $H$ is hyperbolic, any critical point of $H$ must be of saddle-type and in finite number. So, we can extend the dominated splitting and assume that $\Phi_{H}^t$ admits a dominated splitting over $\mathcal{E}^\star_{H,e}\backslash$Crit$(H|_{\mathcal{E}^\star_{H,e}})$. Therefore, directly from Corollary \ref{viviercrit}, we conclude that Crit$(H|_{\mathcal{E}^\star_{H,e}})=\emptyset$ and so that $\Phi_{H}^t$ admits a dominated splitting over $\mathcal{E}^\star_{H,e}$. 

Observe that, since we are assuming $\mathcal{E}^\star_{H,e}$ to be regular, we have a well defined invariant volume measure $\mu_{\mathcal{E}^\star_{H,e}}$ on $\mathcal{E}^\star_{H,e}$.

Now, by contradiction, assume that $\Phi_{H}^t$ does not admit a dominated splitting over $\mathcal{E}^\star_{H,e}$. Then there exist a $\mu_{\mathcal{E}^\star_{H,e}}$-positive measure  and $X_{H}^t$-invariant set $B\subset\mathcal{E}^\star_{H,e}$ such that $B$ does not admit a dominated splitting for $\Phi_H^t$. In this case we claim that
\begin{claim} For every $m\in \mathbb{N}$, there exists a $\mu_{\mathcal{E}^\star_{H,e}}$-positive measure and $X_{H}^t$-invariant subset of $B$, say $\Gamma_m$, such that $\Gamma_m$ does not admit $m$-dominated splitting for $\Phi_{H}^t$.
\end{claim}
If this claim was not true, there would be $m\in \mathbb{N}$ such that all $\Gamma_m$ in the above conditions would admit an $m$-dominated decomposition for $\Phi_{H}^t$. Taking $\Gamma_m:=B$, we would reach a contradiction, since $B$ does not admit a dominated splitting for $\Phi_H^t$. 

Since $(H,e)\in\mathcal{G}$$^{2}(M)$ we have that each Hamiltonian $\epsilon$-$C^2$-close to $H$ does not admit elliptic closed orbits. Then, from Lemma \ref{lemateta}, for every such a Hamiltonian $H$ there are constants $\theta=\theta(\epsilon,H)$, $m=m(\epsilon, \theta)$ and $T=T(m)$ such that, for each closed orbit with period larger than $T$, we can ensure $m$-dominated splitting and that the angle between its stable and unstable directions is bounded from below by $\theta$. Notice that these closed orbits are all hyperbolic.

Since $\mathcal{E}^\star_{H,e}$ is a compact energy surface and $\mu_{\mathcal{E}^\star_{H,e}}$ is $X_H^t$-invariant, we can apply the \textit{Poincar\'{e} Recurrence Theorem}  on $\mathcal{E}_{H,e}$. So, let $R$ be a measurable and $\mu_{\mathcal{E}^\star_{H,e}}$-positive measure subset of $\Gamma_{m}$ given by the \textit{Poincar\'{e} Recurrence Theorem} with respect to $X_{H}|_{\mathcal{E}^\star_{H,e}}$. Then, $\mu_{\mathcal{E}^\star_{H,e}}(R)=\mu_{\mathcal{E}^\star_{H,e}}(\Gamma_m)$ and every $x\in R$ returns to $\Gamma_m$ infinitely many times under the flow $X_{H}^t|_{\mathcal{E}^\star_{H,e}}$. We observe that the set of closed orbits of period less than $k \in \mathbb{N}$  is a set of zero measure. Let $Q$ denote the subset of points of $\Gamma_m$ having zero Lyapunov exponents for $X_{H}$ on $\mathcal{E}^\star_{H,e}$. We want to choose a point $x \in Q\cap R$; if $\mu_{\mathcal{E}^\star_{H,e}}(Q)>0$ we are done. Now, consider the reverse case: $\mu_{\mathcal{E}^\star_{H,e}}$-a.e. $x\in \Gamma_m$ has a nonzero Lyapunov exponent for $X_{H}|_{\mathcal{E}^\star_{H,e}}$, i.e., $\mu_{\mathcal{E}^\star_{H,e}}(Q)=0$. In this case, the idea is to take $x\in R$ and to use the techniques involved in the proof of Theorem \ref{bessa2} (see \cite{MBJLD}), in order to cause the decay of the Lyapunov exponents. So, for $m$ sufficiently large and $\eta>0$ arbitrarily small, there exist $T_0>0$ and $H_1$, $\epsilon$-$C^2$-close to $H$, and $x$ has Lyapunov exponent less than $\eta$ for $X_{H_1}|_{\mathcal{E}^\star_{H_1,e}}$, i.e.
\begin{displaymath}
\exp(-\eta t)<\left\|\Phi_{H_1}^t(x)\right\|< \exp(\eta t), \;\text{for every} \; t> T_0.
\end{displaymath}

Now, fixing $\delta \in \bigl(0, \frac{log(2)}{2m}\bigr)$ and $\eta<\delta$, one has that there is $T_x\in \mathbb{R}$ such that
\begin{displaymath}
\exp(-\delta t)<\left\|\Phi_{H_1}^t(x)\right\|< \exp(\delta t), \;\text{for every} \; t\geq T_x.
\end{displaymath}
Notice that we can assume $T_x\geq T$.

Once  $x\in R$, we may use the Lemma \ref{closing} and conclude that the $X_{H_1}^t$-orbit of $x$ can be approximated, for a very long recurrent time $\tilde{T}>T_x$ by a closed orbit of a $C^1$-close flow $X_{H_2}^t$: given $r, \tilde{T}>0$ we can find a $\epsilon$-$C^2$-neighbourhood $\mathcal{U}$ of $H_1$, a closed orbit $\Gamma$ of $H_2\in\mathcal{U}$ with period $\pi$, as large as we want, $\hat{T}>\tilde{T}$ and $g:\left[0,\tilde{T}\right]\rightarrow \left[0,\pi \right]$ close to the identity such that, for $p\in \Gamma$,
\begin{itemize}
	\item $dist\Bigl(X^t_{H_1}(x), X^{g(t)}_{H_2}(p)\Bigr)<r$, $0\leq t\leq \hat{T}$,
	\item $H_1=H_2$ on $M\backslash \bigcup_{0\leq t\leq \hat{T}}\Bigl( B_r\big(X^t_{H_1}(p)\big)\Bigr)$.
\end{itemize}
Letting $r$ be small enough, we also have that
\begin{equation}\label{exp2}
\exp(-\delta \pi)< \left\|\Phi_{H_2}^{\pi}(p)\right\|<\exp(\delta \pi)
\end{equation} 
where $\pi>T$.

Once, by construction, $H_2$ is $\epsilon$-$C^2$-close to $H$, one has that the orbit of $p$ under $X_{H_2}^t$ satisfies the conclusions of Lemma \ref{lemateta} and that $\left\|\Phi_{H_2}^m(x)|_{{\mathcal{N}}_x^-}\right\|\leq \frac{1}{2}\left\|\Phi_{H_2}^m(x)|_{{\mathcal{N}}_{x}^+}\right\|$, for all $x$ in the $X_{H_2}^t$-orbit of $p$.

Let $p_i=X_{H_2}^{im}(p)$ for $i=0,...,[\pi/m]$, where $[t]:=\max\left\{k\in\mathbb{Z}:k\leq t\right\}$. Since the subbundles $\mathcal{N}^-$ and $\mathcal{N}^+$ are one-dimensional one has

\begin{eqnarray}
\dfrac{\left\|\Phi_{H_2}^{\pi}(p)|_{{\mathcal{N}}_p^-}\right\|}{\left\|\Phi_{H_2}^{\pi}(p)|_{{\mathcal{N}}_p^+}\right\|}&=&
\dfrac{\left\|\Phi_{H_2}^{\pi-m[\pi/m]+m[\pi/m]}(p)|_{{\mathcal{N}}_p^-}\right\|}{\left\|\Phi_{H_2}^{\pi-m[\pi/m]+m[\pi/m]}(p)|_{{\mathcal{N}}_p^+}\right\|}\\
&=&\dfrac{\left\|\Phi_{H_2}^{\pi-m[\pi/m]}(p)|_{{\mathcal{N}}_p^-}\right\|}{\left\|\Phi_{H_2}^{\pi-m[\pi/m]}(p)|_{{\mathcal{N}}_p^+}\right\|}
\cdot \prod_{i=1}^{[\pi/m]} \dfrac{\left\|\Phi_{H_2}^{m}(p_i)|_{{\mathcal{N}}_{p_i}^-}\right\|}{\left\|\Phi_{H_2}^{m}(p_i)|_{{\mathcal{N}}_{p_i}^+}\right\|}\nonumber\\
&\leq& C(p, H_2)\cdot \biggl(\dfrac{1}{2}\biggr) ^{[\pi/m]}\label{exp1}
\end{eqnarray}
where $$C(p, H_2)=\sup\left\{0\leq t \leq m: \left\|\Phi_{H_2}^t(p)|_{{\mathcal{N}}_p^-}\right\|\cdot \left\|\Phi_{H_2}^t(p)|_{{\mathcal{N}}_p^+}\right\|^{-1}\right\}$$ depends continuously on $H_2$ in the $C^2$ topology. Then, there exists a uniform bound for $C(p,\cdot)$ for all Hamiltonians which are $C^2$-close to $H$.

If we let $r>0$ be small enough, we can take $\pi>T$ arbitrarily large. So, inequality (\ref{exp1}) ensures that
\begin{displaymath}
\dfrac{1}{\pi}\log \left\|\Phi_{H_2}^{\pi}(p)|_{{\mathcal{N}}_p^-}\right\|  \leq \dfrac{1}{\pi}\log C(p,H_2)+ \dfrac{[\pi/m]}{\pi}\log\dfrac{1}{2}+\dfrac{1}{\pi}\log \left\|\Phi_{H_2}^{\pi}(p)|_{{\mathcal{N}}_p^+}\right\|
\end{displaymath}
and also  $\left\|\Phi_{H_2}^{\pi}(p)\right\|=\left\|\Phi_{H_2}^{\pi}(p)|_{{\mathcal{N}}_x^+}\right\|$. Moreover, since $\Phi_{H_2}^{\pi}$ is conservative one has that the sum of the Lyapunov exponents is zero, i.e.

\begin{displaymath}
\dfrac{1}{\pi}\log \left\|\Phi_{H_2}^{\pi}(p)|_{{\mathcal{N}}_p^-}\right\| = - \dfrac{1}{\pi}\log \left\|\Phi_{H_2}^{\pi}(p)|_{{\mathcal{N}}_p^+}\right\| 
\end{displaymath}
So, one has that
\begin{align*}
\dfrac{2}{\pi}\log \left\|\Phi_{H_2}^{\pi}(p)\right\| =\dfrac{2}{\pi}\log \left\|\Phi_{H_2}^{\pi}(p)|_{{\mathcal{N}}_p^+}\right\|&\geq 
-\dfrac{1}{\pi} \log C(p,H_2)-\dfrac{[\pi/m]}{\pi}\log\dfrac{1}{2}\nonumber\\
&\geq -\dfrac{1}{\pi} \log C(p,H_2) +\dfrac{1}{m}\log 2\nonumber.
\end{align*}
Notice that the constants involved in inequality (\ref{exp1}) do not depend on $\pi$. So, we can take the period of $p$ very large such that
\begin{displaymath}
\dfrac{1}{\pi}\log \left\|\Phi_{H_2}^{\pi}(p)\right\| \geq \dfrac{1}{2m}\log 2>\delta. 
\end{displaymath}
This contradicts expression (\ref{exp2}). Then we have that $\Phi_H^t$ admits a do\-mi\-nated splitting over $\mathcal{E}^\star_{H,e}$.
\end{proof}

\begin{remark}\label{remaelliptic}
It follows from the previous proof that the conclusion of Lemma~\ref{lemasing} also holds if we assume that the Hamiltonian flow is far from elliptic orbits and the energy surface $\mathcal{E}^\star_{H,e}$ is regular.
\end{remark}

\bigskip
\begin{lemma}\label{lemanosov}
If $(H,e)\in C^2(M,\mathbb{R})\times H(M)$ is such that $\Phi_{H}^t$ admits a dominated splitting over a regular energy surface $\mathcal{E}_{H,e}$ then $(H,e)$ is an Anosov Hamiltonian level.
\end{lemma}

\begin{proof}
Once that $\Phi_{H}^t$ admits a dominated splitting over $\mathcal{E}_{H,e}$, we have that there exists $m\in \mathbb{N}$ and a constant $\theta\in(0,1)$ such that
$$
\Delta(x,m):= \|\Phi_H^m(x)|_{{\mathcal{N}}_x^-}\|\: \|\Phi_H^{-m}\bigl(X_H^m(x)\bigr)|_{\mathcal{N}_{X_H^m(x)}^+}\|\leq\theta, \ \ \forall \: x\in \mathcal{E}_{H,e}.
$$

Due to the chain rule, for any $i\in \mathbb{N}$ we have $\Delta(x,im)\leq\theta^i$. Also, every $t\in \mathbb{R}$ can be written as $t=im+r$, where $r\in [0,m)$. Since the manifold $M$ is compact, we have that $\left\|\Phi_H^r\right\|$ is bounded, say by $L$. So, we can take $C:=\theta^{-\frac{r}{m}}L^2$ and $\sigma:=\theta^{\frac{1}{m}}$. As $C>0$ and $0<\sigma<1$ these are good candidates to be the constants of hyperbolicity. In fact, for every $x\in \mathcal{E}_{H,e}$ and $t \in\mathbb{R}$, one has that
\begin{align}
\Delta(x,t)&=\Delta(x,im+r)=\|\Phi_H^{im+r}(x)|_{{\mathcal{N}}_x^-}\|\|\Phi_H^{-im-r}\bigl(X_H^{im+r}(x)\bigr)|_{{\mathcal{N}}_{X_H^{im+r}(x)}^+}\|\nonumber\\
&=\|\Phi_H^{im}\bigl(X_H^r(x)\bigr)|_{{\mathcal{N}}_{X^r_H(x)}^-}\|\|\Phi_H^{r}(x)|_{{\mathcal{N}}_x^-}\|\:.\\
&\,\,\,\,\,\,. \,\|\Phi_H^{-im}\bigl(X_H^{im}(x)\bigr)|_{{\mathcal{N}}_{X^{im}_H(x)}^+}\| \|\Phi_H^{-r}\bigl(X_H^{im+r}(x)\bigr)|_{{\mathcal{N}}_{X^{im+r}_H(x)}^+}\|\nonumber\\
&\leq L^2\: \Delta(x,im)\leq L^2\: \theta^i= L^2 \:\theta^{\frac{im}{m}}=L^2 \:\theta^{\frac {t-r}{m}}= \theta^{\frac{-r}{m}}L^2\:\theta^{\frac{t}{m}}=C\:\sigma^t\nonumber.
\end{align}

Denote by $\alpha_t$ the angle between the fibers $\mathcal{N}$$^-_{X_H^t(x)}$ and $\mathcal{N}$$^+_{X_H^t(x)}$ and notice that, by domination, $\alpha_t\geq \beta>0$. Since Crit$(H|\mathcal{E}_{H,e})=\emptyset$, there is $K>1$ such that, for every $x\in \mathcal{E}_{H,e}$, $K^{-1}\leq\left\|X_H(x)\right\|\leq K$. As $\Phi_H^t$ is conservative and  the subbundles $\mathcal{N}$$^-$ and $\mathcal{N}$$^+$ are both one dimensional, we have that
\begin{displaymath}
\sin(\alpha_0)=\left\|\Phi_H^t(x)|_{{\mathcal{N}}_x^-}\right\| \left\|\Phi_H^t(x)|_{{\mathcal{N}}_x^+}\right\| \sin(\alpha_t)\dfrac{\left\|X_H(X_H^t(x))\right\|}{\left\|X_H(x)\right\|}.
\end{displaymath}

Given $t \in \mathbb{R}$, as $0<\beta\leq \alpha_t<\dfrac{\pi}{2}$, we have that ${\sin(\alpha_t)}\geq {\sin(\beta)}$. Taking a positive $C_1:=\sin(\beta)^{-1}\: K^2$, for every $x \in \mathcal{E}_{H,e}$ we have that
\begin{align}
\left\|\Phi_H^t(x)|_{{\mathcal{N}}_x^+}\right\|^2&=\dfrac{\sin(\alpha_0)}{\sin(\alpha_t)} \; \dfrac{\left\|X_H(x)\right\|}{\left\|X_H(X_H^t(x))\right\|}   \left\|\Phi_H^{-t}(x)|_{{\mathcal{N}}_x^-}\right\| \left\|\Phi_H^t(x)|_{{\mathcal{N}}_x^+}\right\|\nonumber\\
&\leq \sin(\beta)^{-1}\: K^2\;\Delta(x,t)\leq \sin(\beta)^{-1}\: K^2\; C\: \sigma^t     \nonumber\\
&= C_1\:\sigma^t\nonumber.
\end{align}

Analogously, for every $x\in \mathcal{E}_{H,e}$ we get 
\begin{align}
\left\|\Phi_H^{-t}(x)|_{{\mathcal{N}}_x^-}\right\|^2&=\dfrac{\sin(\alpha_t)}{\sin(\alpha_0)} \; \dfrac{\left\|X_H(X_H^t(x))\right\|}{\left\|X_H(x)\right\|} \;\Delta(x,t)\nonumber\\
&\leq \sin(\beta)^{-1}\: K^2\; C\: \sigma^t     \nonumber\\
&= C_1\:\sigma^t\nonumber .
\end{align}

These two inequalities show that $\mathcal{E}_{H,e}$ is hyperbolic for the transversal linear Poincar\'{e} flow. Then, by Lemma \ref{hyperbolicpoinc}, $\mathcal{E}_{H,e}$ is also hyperbolic for $X_H^t$, i.e., $(H,e)$ is Anosov.
\end{proof}

This finishes the first part of Theorem~\ref{mainth}. 

\begin{remark}
Notice that this result is more general than it looks because, since we do not need the energy $e$ to vary, it works for a larger number of Hamiltonian levels $(H,e)$.

We also point out that, if each one of the finite connected components of $H^{-1}(\{e\})$ are regular and belong to $\mathcal{G}^2(M)$, we are able to conclude that all the energy levels $H^{-1}(\{e\})$ are Anosov.
\end{remark}

Now, we easily derive the remaining part of Theorem \ref{mainth}.

\begin{lemma}
If $(H,e)\in\mathcal{G}$$^{2}(M)$ is such that $\mathcal{E}^\star_{H,e}$ is not regular, then there exists $\tilde{e}$, arbitrarily close to $e$, such that $X_H^t|_{\mathcal{E}^\star_{H,\tilde{e}}}$ is Anosov.
\end{lemma}

\begin{proof} The proof is straightforward. As $(H,e)\in\mathcal{G}$$^{2}(M)$, $\mathcal{E}_{H,e}$ has only a finite number of  critical points and they are hyperbolic for $\Phi_H^t$. Therefore there exists $\tilde{e}$, arbitrarily close to $e$, such that $(H,\tilde{e})\in\mathcal{G}$$^{2}(M)$ and $\mathcal{E}^\star_{H,\tilde{e}}$ is regular. Now previous lemmas show that $(H,\tilde{e})$ is Anosov.
\end{proof}

\end{subsection}

%%%%%%%%%%%%%%%%%%%%%%%%%%%%%%%%%%%%%%%%%%%%%%%%%%%%%%%%%%%%%%%%%%%%%%%%%%%%%%%%%%%%%%%%%%%%%%%%%%%%%%%%%%%%%%%%%%%%%%%%%%%%%%%%%%%%%%%%%%%%%%%%%%%

As a consequence of Theorem \ref{mainth} we prove that the boundary of $\mathcal{A}^4(M)$ has no isolated points.

\begin{proof}(of Corollary~\ref{maincorollary})

By contradiction, let $(H,e)$ be an isolated point on the boundary of $\mathcal{A}^4(M)$. We start the prove by claiming:
\begin{claim}\label{cl}
If $(H,e)$ is an isolated point on the boundary of $\mathcal{A}^4(M)$ then any energy surface  $\mathcal{E}_{H,e}$ is regular.
\end{claim}

If this claim was not true, we could take a critical point $q$ associated to some energy surface $\mathcal{E}_{H,e}$. It could be hyperbolic, or not. If $q$ is hyperbolic then, since $(H,e)$ is isolated on the boundary of $\mathcal{A}^4(M)$, an adequate perturbation of $(H,e)$ will produce a Hamiltonian level $(\tilde{H},\tilde{e})$ on $\mathcal{A}^4(M)$ with a critical point, which is a contradiction. Now, supposing that $q$ is not hyperbolic by a small adequate perturbation on $(H,e)$ we can make it hyperbolic, which again is a contradiction because $(H,e)$ is an isolated point of the boundary of $\mathcal{A}^4(M)$. This proves the claim.

Now, we fix some energy level $\mathcal{E}_{H,e}$ and follow the ideas presented in the proof of Theorem \ref{mainth} to get a contradiction. We start by proving that $\Phi_H^t$ admits a dominated splitting over $\mathcal{E}_{H,e}$. Notice that, in Lemma \ref{lemasing}, the main step is obtained because we had $(H,e)\in \mathcal{G}^2(M)$, and so elliptic orbits are not allowed in $\mathcal{E}_{H,e}$. However, even without this assumption, we can go on with a similarly proof because $(H,e)$ is an isolated point on the boundary of $\mathcal{A}^4(M)$. So, any small perturbation $(\tilde{H}, \tilde{e})$ arbitrarily close to $(H,e)$ will be in $\mathcal{A}^4(M)$, which enables the existence of elliptic orbits in $\mathcal{E}_{H,e}$. Finally, since the Claim \ref{cl} ensures that the energy surface $\mathcal{E}_{H,e}$ is regular, by Lemma \ref{lemanosov} we get that $\Phi_H^t|_{\mathcal{E}_{H,e}}$, in particular the Hamiltonian level $(H,e)$ is Anosov. This is a contradiction because we took $(H,e)$ on the boundary of $\mathcal{A}^4(M)$. So, the boundary of $\mathcal{A}^4(M)$ can not have isolated points.
\end{proof}

%%%%%%%%%%%%%%%%%%%%%%%%%%%%%%%%%%%%%%%%%%%%%%%%%%%%%%%%%%%%%%%%%%%%%%%%%%%%%%%%%%%%%%%%%%%%%%%%%%%%%%%%%%%%%%%%%%%%%%%%%%%%%%%%%%%%%%%%%%%%%%%%%%%
%%%%%%%%%%%%%%%%%%%%%%%%%%%%%%%%%%%%%%%%%%%%%%%%%%%%%%%%%%%%%%%%%%%%%%%%%%%%%%%%%%%%%%%%%%%%%%%%%%%%%%%%%%%%%%%%%%%%%%%%%%%%%%%%%%%%%%%%%%%%%%%%%%%
\end{section}

%%%%%%%%%%%%%%%%%%%%%%%%%%%%%%%%%%%%%%%%%%%%%%%%%%%%%%%%%%%%%%%%%%%%%%%%%%%%%%%%%%%%%%%%%%%%%%%%%%%%%%%%%%%%%%%%%%%%%%%%%%%%%%%%%%%%%%%%%%%%%%%%%%%
%%%%%%%%%%%%%%%%%%%%%%%%%%%%%%%%%%%%%%%%%%%%%%%%%%%%%%%%%%%%%%%%%%%%%%%%%%%%%%%%%%%%%%%%%%%%%%%%%%%%%%%%%%%%%%%%%%%%%%%%%%%%%%%%%%%%%%%%%%%%%%%%%%%
%%%%%%%%%%%%%%%%%%%%%%%%%%%%%%%%%%%%%%%%%%%%%%%%%%%%%%%%%%%%%%%%%%%%%%%%%%%%%%%%%%%%%%%%%%%%%%%%%%%%%%%%%%%%%%%%%%%%%%%%%%%%%%%%%%%%%%%%%%%%%%%%%%%

\begin{section}{Anosov Hamiltonian levels}\label{anosov}

%%%%%%%%%%%%%%%%%%%%%%%%%%%%%%%%%%%%%%%%%%%%%%%%%%%%%%%%%%%%%%%%%%%%%%%%%%%%%%%%%%%%%%%%%%%%%%%%%%%%%%%%%%%%%%%%%%%%%%%%%%%%%%%%%%%%%%%%%%%%%%%%%%%
\begin{subsection}{Openness of $\mathcal{A}^{2d}(M)$ \label{open}}

The next result states that Anosov Hamiltonian levels are open on $C^2(M,\mathbb{R})\times\mathbb{R}$. The proof follows standard cone-fields arguments that can be found, for example,  in the book of Brin and Stuck (\cite{brin}).

Let $\Lambda \subset \mathcal{R}$ be a hyperbolic set for $\Phi_H^t$. Since the subbundles $\mathcal{N}^-$ and $\mathcal{N}^+$ are continuous, we extend them to continuous subbundles $\mathcal{\tilde{N}}^-$ and $\mathcal{\tilde{N}}^+$, defined on a neighbourhood $\mathcal{U}$ of $\Lambda$, $\mathcal{U}\subset \mathcal{R}$. Take $x\in \mathcal{U}$ and $v\in \mathcal{N}_x$ and let $v=v^-+v^+$ with $v^-\in\mathcal{N}^-_x$ and $v^+\in\mathcal{N}^+_x$. For $\alpha>0$, define the \textit{stable} and \textit{unstable cones} of size $\alpha$ by  

\begin{displaymath}
K_{\alpha}^-(x)=\left\{v\in \mathcal{N}_x:     \left\|v^+\right\|   \leq \alpha\left\|v^-\right \|   \right\},
\end{displaymath}
\begin{displaymath}
K_{\alpha}^+(x)=\left\{v\in \mathcal{N}_x:     \left\|v^-\right\|   \leq \alpha\left\|v^+\right \|   \right\}.
\end{displaymath}

We have the following proposition.
\bigskip

\begin{proposition}\label{propcone}
Take $H\in C^2(M, \mathbb{R})$ and $\Lambda\subset M \setminus \text{Crit}(H)$ a compact and $X_H^t$-invariant set. Suppose that there are $m\in \mathbb{N}$, $\alpha>0$ and continuous subspaces $\mathcal{\tilde{N}}^-_x$ and $\mathcal{\tilde{N}}^+_x$, for every $x\in \Lambda$, such that $\mathcal{N}_x=\mathcal{\tilde{N}}^-_x\oplus\mathcal{\tilde{N}}^+_x$ and that the $\alpha$-cones $K_{\alpha}^-(x)$ and $K_{\alpha}^+(x)$, determined by the subspaces, satisfy
\begin{itemize}
	\item $\Phi_H^t(x)\bigl(K_{\alpha}^+(x)\bigr)\subset K_{\alpha}^+\bigl(X_H^t(x)\bigr)$, $t\geq 0$,
	\item $\Phi_H^{-t}\bigl(X_H^t(x)\bigr)\bigl(K_{\alpha}^-\bigl(X_H^t(x)\bigr)\bigr)\subset K_{\alpha}^-(x)$, $t\geq 0$,
	\item $\left\| \Phi_H^m(x)v \right\|< \left\| v \right\|, \forall\: v \in K_{\alpha}^-(x)\backslash \left\{0\right\}$,
	\item $\left\| \Phi_H^{-m}(x)v \right\|< \left\| v \right\|, \forall\: v \in K_{\alpha}^+(x)\backslash \left\{0\right\}$.
\end{itemize}
Then $\Lambda$ is a hyperbolic set for $\Phi_H^t$.
\end{proposition}

\begin{proof}
By compactness of $\Lambda$ and of the unit tangent bundle of $M$, there is a constant $\theta\in(0,1)$ such that $\left\| \Phi_H^m(x)v \right\|\leq \theta\left\| v \right\|, \forall \:v \in K_{\alpha}^-(x)$ and $\left\| \Phi_H^{-m}(x)v \right\|\leq \theta\left\| v \right\|, \forall \: v \in K_{\alpha}^+(x)$.

For any $x\in \Lambda$, define 

$$
\mathcal{N}_x^-:=\bigcap_{n\in \mathbb{N}_0}\Phi_H^{-n}\bigl(X_H^{n}(x)\bigr)K_{\alpha}^-\bigl(X_H^{n}(x)\bigr) \ \ \ \ \text{ and}$$

$$ \mathcal{N}_x^+:=\bigcap_{n\in \mathbb{N}_0}\Phi_H^{n}\bigl(X_H^{-n}(x)\bigr)K_{\alpha}^+\bigl(X_H^{-n}(x)\bigr).
$$

Obviously we have that $\mathcal{N}_x=\mathcal{N}^-_x\oplus\mathcal{N}^+_x$ and that the fibers are invariant under the flow. Also, notice that $\mathcal{N}_x^-\subset K_\alpha^-(x)$ and $\mathcal{N}_x^+\subset K_\alpha^+(x)$. So, $\left\| \Phi_H^m(x)v \right\|\leq \theta\left\| v \right\|, \forall \:v \in \mathcal{N}_x^-$ and  $\left\| \Phi_H^{-m}(x)v \right\|\leq \theta\left\| v \right\|, \forall \:v \in \mathcal{N}_x^+$.
This means that $\Lambda$ is a hyperbolic set for $\Phi_H^t$.
\end{proof}

\begin{maintheorem}\label{Aopen}
$\mathcal{A}^{2d}(M)$ is open in $C^2(M, \mathbb{R})\times H(M)$.
\end{maintheorem}

\begin{proof}
We want to prove that if $(H,e)\in C^2(M, \mathbb{R})\times H(M)$ is an Anosov Hamiltonian level then there is a $C^2$-neighbourhood $\mathcal{V}$ of $H$ and $\delta>0$ such that, for every $\tilde{H}$ in $\mathcal{V}$ and every $\tilde{e}\in \left(e-\delta, e+\delta\right)$, the Hamiltonian level $(\tilde{H},\tilde{e})$ is also Anosov.

Since the Hamiltonian level $(H,e)$ is Anosov, we have that there ex\-ists a (compact and $X_H^t$-invariant) energy surface $\mathcal{E}_{H,e}$ which is hyperbolic. So, we have the $\Phi_H^t$-invariant and hyperbolic splitting $\mathcal{N}_{\mathcal{E}_{H,e}}=\mathcal{N}_{\mathcal{E}_{H,e}}^+\oplus \mathcal{N}_{\mathcal{E}_{H,e}}^-$. 

We have seen that, fixing a Hamiltonian function, $H^{-1}(\{e\})$ can be splitted into a finite number of disjoint compact connected components. These components are pairwise separated by a positive distance. So, since the compact regular energy surface $\mathcal{E}_{H,e}$ is a connected component of the set $H^{-1}(\left\{e\right\})$, for any small neighbourhood $\mathcal{U}$ of $\mathcal{E}_{H,e}$, there is a $C^2$-neighbourhood $\mathcal{V}$ of $H$ and $\delta>0$ such that $\forall\; \tilde{e}\in (e-\delta,e+\delta)$ and $\forall\; \tilde{H}\in \mathcal{V}$ such that the set $\tilde{H}^{-1}(\left\{\tilde{e}\right\})$ admits exactly one connected component in $\mathcal{U}$, say $\mathcal{E}_{\tilde{H},\tilde{e}}$.

Now, we continuously extend $\mathcal{N}^-$ and $\mathcal{N}^+$ over $\mathcal{E}_{H,e}$ to $\mathcal{\tilde{N}}^-$ and $\mathcal{\tilde{N}}^+$ over $\mathcal{U}$. For an appropriate choice of small $\mathcal{U}$ and $\alpha>0$, we have that, for every $(\tilde{H},\tilde{e})\in \mathcal{V}\times (e-\delta,e+\delta)$, the stable and unstable $\alpha$-cones determined by $\mathcal{\tilde{N}}^-$ and $\mathcal{\tilde{N}}^+$ satisfy the assumptions of Proposition \ref{propcone} for $\Phi_{\tilde{H}}^t$ on $\mathcal{E}_{\tilde{H},\tilde{e}}$. So, we have that $\mathcal{E}_{\tilde{H},\tilde{e}}$ is hyperbolic for $\Phi_{\tilde{H}}^t$, i.e., the Hamiltonian level $(\tilde{H},\tilde{e})$ is Anosov.
\end{proof}

\end{subsection}

%%%%%%%%%%%%%%%%%%%%%%%%%%%%%%%%%%%%%%%%%%%%%%%%%%%%%%%%%%%%%%%%%%%%%%%%%%%%%%%%%%%%%%%%%%%%%%%%%%%%%%%%%%%%%%%%%%%%%%%%%%%%%%%%%%%%%%%%%%%%%%%%%%%
%%%%%%%%%%%%%%%%%%%%%%%%%%%%%%%%%%%%%%%%%%%%%%%%%%%%%%%%%%%%%%%%%%%%%%%%%%%%%%%%%%%%%%%%%%%%%%%%%%%%%%%%%%%%%%%%%%%%%%%%%%%%%%%%%%%%%%%%%%%%%%%%%%%

\begin{subsection}{Structural stability}\label{strucstab}

A flow $X^t\colon M\rightarrow M$ is a \textit{time change} of another flow $Y^t\colon M\rightarrow M$ if for each $x\in M$ the orbits $\left\{X^t(x)\right\}_{t\in\mathbb{R}}$ and $\left\{Y^t(x)\right\}_{t\in\mathbb{R}}$ coincide and the orientations given by the change of $t$ in the positive direction are the same. It means that $X^t(x)=Y^{\alpha(t,x)}(x)$ for every $x\in M$, where $\alpha$ is a real-valued  function such that $\alpha(0,x)=0$ and $\alpha(\cdot,x)$ is increasing.

Two flows $X^t\colon M_1\rightarrow M_1$ and $Y^t\colon M_2\rightarrow M_2$ are said to be \textit{orbit equivalent} if there is a homeomorphism $h\colon M_1\rightarrow M_2$ such that the flow $h^{-1}\circ\ Y^t\circ h$ is a time change of the flow $X^t$. 

We say that a Hamiltonian level $(H,e)\in C^2(M, \mathbb{R})\times H(M)$ is \textit{structurally stable} if there is a $C^2$-neighbourhood $\mathcal{U}$ of $H$ and $\delta>0$ such that, for every $\tilde{H}$ in $\mathcal{U}$ and every $\tilde{e}\in \left(e-\delta, e+\delta\right)$, the flows $X_H^t|_{\mathcal{E}_{H,e}}$ and $X_{\tilde{H}}^t|_{\mathcal{E}_{\tilde{H},\tilde{e}}}$ are orbit equivalent, where ${\mathcal{E}_{H,e}}$ is a fixed energy surface and ${\mathcal{E}_{\tilde{H},\tilde{e}}}$ is the analytic continuation of it.

If, in addition, the homeomorphism in question can be chosen close enough to the identity for small perturbations, then we say that $(H,e)$ is \textit{strongly structurally stable}.

\begin{theorem} (\cite[Theorem 18.2.3]{Katok})\label{thkatok}
Let $\Lambda \subset M$ be a hyperbolic set of the smooth flow $X^t$ on $M$. Then for any open neighbourhood $\mathcal{V}$ of $\Lambda$ and every $\delta>0$ there exists $\epsilon>0$ such that if $Y^t$ is another smooth flow and $\|X-Y\|_{C^1}<\epsilon$ then there is an invariant hyperbolic set $\Lambda'$ for $Y^t$ and a homeomorphism $h:\Lambda \rightarrow \Lambda'$ with $dist(Id,h)+dist(Id,h^{-1})<\delta$ that is smooth along the orbits of $X^t$ and maps orbits of $X^t$ to orbits of $Y^t$, and establishes an orbit equivalence of $X^t$ and $Y^t$. Furthermore, if $h_1$ and $h_2$ are two such homeomorphisms then $h_2^{-1}\circ h_1$ is a time change of $X^t$, close to the identity.
\end{theorem}

In the Hamiltonian setting this result means that if a regular energy surface ${\mathcal{E}_{H,e}}$ is hyperbolic for $\Phi_H^t$ then there are a $C^2$-neighbourhood $\mathcal{U}$ of $H$ and $\delta>0$ such that for every $\tilde{H} \in \mathcal{U}$ and $\tilde{e} \in (e-\delta, e+\delta)$ such that ${\mathcal{E}_{\tilde{H},\tilde{e}}}$ is hyperbolic for $\Phi_{\tilde{H}}^t$,  and there exists a homeomorphism $h \colon {\mathcal{E}_{H,e}} \rightarrow {\mathcal{E}_{\tilde{H},\tilde{e}}}$ with the properties described in the theorem above.

In dimension four, we are  able to prove that, for Hamiltonian levels, the notions of structural stability and of Anosov system are equivalent.

%\begin{theorem} (\cite[Theorem 2]{MBJLD2})\label{mario2} For a $C^2$-generic Hamiltonian $H\in C^s(M, \mathbb{R})$, $2\leq s \leq \infty$, the union of the Anosov regular energy surfaces and the closed elliptic orbits forms a dense subset of $M$.\end{theorem}

So, we have that

\begin{maintheorem}\label{ssA}
If $(H,e)\in C^2(M, \mathbb{R})\times H(M)$ is a structurally stable Hamiltonian level then $(H,e)\in\mathcal{A}^{4}(M)$.
\end{maintheorem}

\begin{proof}
Fix a structurally stable Hamiltonian level $(H,e)$ and take $\mathcal{U}$ a neighbourhood of $H$ and $\delta>0$ such that, for every $\tilde{H}\in\mathcal{U}$ and for every $\tilde{e}\in(e-\delta,e+\delta)$, we have that $(\tilde{H},\tilde{e})$ topologically equivalent to $(H,e)$. In particular one has  that $\tilde{H}\in\mathcal{U}$ is regular because  critical points can be easily destroyed by small perturbations of the energy.

By contradiction, suppose that $(H,e)$ is not an Anosov system therefore, using Lemma~\ref{lemanosov}, none of the energy surfaces associated to $e$, ${\mathcal{E}_{H,e}}$, admit a dominated splitting. From Remark~\ref{remaelliptic} one gets that $(H,e)$ can be approximated by  $(\tilde{H},\tilde{e})$ such that $\mathcal{E}_{\tilde{H},\tilde{e}}$ has an elliptic closed orbit; moreover, it follows from the proof of Lemma~\ref{lemasing} that this orbit can be chosen with period arbitrarily large.

Now, applying Lemma \ref{Frank} several times, by concatenating small rotations, in order to get a new Hamiltonian level $(\bar{H},\bar{e})$, close to $(H,e)$, and exhibiting a parabolic closed orbit. Let us know formalize this argument.

Let us assume that $p$ is an elliptic closed orbit of (arbitrarily large integer) period $\tilde{\pi}$ and $\theta\in[0,\pi/2]$ is such that $\rho=\exp(\theta i)$ is one eigenvalue of $\Phi_H^{\tilde{\pi}}(p)$. Fix $\epsilon>0$ and $\tau>0$ and let $\delta>0$ be given by Lemma \ref{Frank}. We assume that the period is such that $\tilde{\pi}=\dfrac{\theta}{\alpha}$, where $0<\alpha<\delta$.

Recall that the special linear group $SL(2,\mathbb{R})$ is the group of all real $2\times 2$ matrices with determinant of modulus equal to $1$ and notice that, once we are in the two-dimensional case, the symplectic setting is nothing more than the conservative one. So, let $R_\alpha$ be the rotation matrix of angle $\alpha$, where $\alpha$ is chosen such that $R_\alpha$ is $C^0$-close to the identity. We observe that $\Phi_H^{\tilde{\pi}}(p)$ can be seen as $R_{\theta}$.

By Lemma \ref{Frank}, for $i=1,...,\tilde{\pi}$, for any flowbox $V_i$ of an injective arc of orbit $X_{H}^{[i-1,i]}(p)$ and for a transversal symplectic $\delta$-perturbation $F_i$ of $\Phi_{{H}}^1(X_H^{i-1}(p))$, there exists $H_i\in C^2(M,\mathbb{R})$ satisfying:
\begin{itemize}
	\item $H_i$ is $C^2$-close to $H$,
	\item $\Phi_{{H_i}}^1(X^{i-1}_H(p))=F_i$,
	\item $H=H_i$ on $X_{H}^{[0,1]}(X^i_H(p))\cup (M\backslash V_i)$.
\end{itemize}
Take $$F_i:= \Phi_H^{i}(p)\circ R_{-\alpha}\circ\bigl[\Phi_H^{i-1}\bigl( p\bigr)\bigr]^{-1}$$ and note that $F_i$ is symplectic, since $det\; F_i=1$. 
We define $\tilde{H}=H$, in $M\setminus \bigcup_{i=1}^{\tilde{\pi}}V_i$, and $\tilde{H}=H_i$ in $V_i$, for $i \in\{1,...,\tilde{\pi}\}$.

Now, observe that

\begin{eqnarray*}
\Phi_{{\tilde{H}}}^{\tilde{\pi}}(p)&=&F_{\tilde{\pi}}\circ F_{\tilde{\pi}-1}\circ\cdot\cdot\cdot\circ F_2\circ F_1\nonumber
=  \Phi_{\tilde{H}}^{\tilde{\pi}}(p)\circ R_{-\tilde{\pi}\alpha}\\
&=&\Phi_{\tilde{H}}^{\tilde{\pi}}(p)\circ R_{-\theta}=Id.\nonumber
\end{eqnarray*}

This is a contradiction, since the presence of a parabolic orbit prevents structural stability. So, we have proved that the Hamiltonian level $(H,e)$ has to be Anosov. 
\end{proof}

Notice that Robinson (\cite[Theorem 6.4]{Robinson}), whilst using different techniques, also proved that the existence of an elliptic periodic point prevents structural stability.

Next result states that Anosov Hamiltonian levels are strongly structurally stable.

\begin{maintheorem}\label{Ass}
If $(H,e)\in\mathcal{A}^{2d}(M)$, $d\geq 2$, then $(H,e)$ is strongly structurally stable.
\end{maintheorem}

\begin{proof}
Take $(H,e)\in\mathcal{A}^{2d}(M)$, i.e., $\mathcal{E}_{H,e}$ is hyperbolic, so regular, for $X^t_H|_{\mathcal{E}_{H,e}}$. As we have seen, for any small neighbourhood $\mathcal{U}$ of $\mathcal{E}_{H,e}$ there is a $C^2$-neighbourhood $\mathcal{V}$ of $H$ and $\delta>0$ such that for every $\tilde{H}\in \mathcal{V}$ and for every $\tilde{e}\in (e-\delta,e+\delta)$ we have that $\tilde{H}^{-1}(\left\{\tilde{e}\right\})$ has exactly one connected component in $\mathcal{U}$, say $\mathcal{E}_{\tilde{H},\tilde{e}}$, that is obviously $X_{\tilde{H}}^t$-invariant. So, taking $\mathcal{U}$ sufficiently small, for every $\tilde{\delta}>0$ one gets $\tilde{\epsilon}>0$ such that $\|H-\tilde{H}\|_{C^2}<\tilde{\epsilon}$. Then, by Theorem \ref{thkatok}, there is a compact $X_{\tilde{H}}^t$-invariant hyperbolic set $\tilde{\Lambda}$ for ${X^t_{\tilde{H}}}|_{\tilde{\Lambda}}$ on $\mathcal{U}$ and a homeomorphism $h:\mathcal{E}_{H,e}\rightarrow \tilde{\Lambda}$, with $dist(Id,h)+dist(Id,h^{-1})<\tilde{\delta}$, that maps orbits of $X_H^t$ to orbits of $X_{\tilde{H}}^t$, preserving its orientation.

We have that $X_H^t|_{\mathcal{E}_{H,e}}$ is hyperbolic and $\mu_{\mathcal{E}_{H,e}}$-conservative. So, by the  \textit{Anosov theorem} (see \cite{Anosov}) $\mu_{\mathcal{E}_{H,e}}$ is ergodic. Now, due to the compactness of $M$, we can conclude that $\mathcal{E}_{H,e}$ has a dense orbit and so, since $h$ takes orbits into orbits, we also have a dense orbit in $\tilde{\Lambda}$. Hence, densely, we have that the image by $H$ of the points in $\tilde{\Lambda}$ is constant. Now, extending to the closure, we can find $\tilde{e}\in (e-\delta,e+\delta)$ such that $\tilde{\Lambda}\subset\mathcal{E}_{\tilde{H},\tilde{e}}$. 
Now, by openness, we have that $\mathcal{E}_{\tilde{H},\tilde{e}}$ is still Anosov. Therefore, using \textit{Anosov's theorem} again, we have that $\mu_{\mathcal{E}_{\tilde{H},\tilde{e}}}$ is ergodic. So, since $\tilde{\Lambda}\subset\mathcal{E}_{\tilde{H},\tilde{e}}$ is compact and $X_{\tilde{H}}^t$-invariant, we must have $\mu_{\mathcal{E}_{\tilde{H},\tilde{e}}}(\tilde{\Lambda})=1$ or $\mu_{\mathcal{E}_{\tilde{H},\tilde{e}}}(\tilde{\Lambda})=0$. If the first case holds, by compactness, we have that $\tilde{\Lambda}=\mathcal{E}_{\tilde{H},\tilde{e}}$. On the other hand, supposing that $\mu_{\mathcal{E}_{\tilde{H},\tilde{e}}}(\tilde{\Lambda})=0$, by Theorem \ref{szpil}, we must have that $\dim(\mathcal{E}_{\tilde{H},\tilde{e}})<2d-1$. However, $\dim(\mathcal{E}_{H,e})=2d-1$ and so, since $h$ preserves the topological dimension, we reach a contradiction. This means that $(H,e)$ is strongly structurally stable, and so structurally stable.
\end{proof}

\end{subsection}

%%%%%%%%%%%%%%%%%%%%%%%%%%%%%%%%%%%%%%%%%%%%%%%%%%%%%%%%%%%%%%%%%%%%%%%%%%%%%%%%%%%%%%%%%%%%%%%%%%%%%%%%%%%%%%%%%%%%%%%%%%%%%%%%%%%%%%%%%%%%%%%%%%%
\end{section}

%%%%%%%%%%%%%%%%%%%%%%%%%%%%%%%%%%%%%%%%%
\section*{Acknowledgements}

M\'ario Bessa was partially supported by Funda\c c\~ao para a Ci\^encia e a Tecnologia, SFRH/BPD/
20890/2004. C\'elia Ferreira was supported by Funda\c c\~ao para a Ci\^encia e a Tecnologia, SFRH/BD/
33100/2007.

%%%%%%%%%%%%%%%%%%%%%%%%%%%%%%%%%%%%%%%%%

%%%%%%%%%%%%%%%%%%%%%%%%%%%%%%%%%%%%%%%%%%%%%%%%%%%%%%%%%%%%%%%%%%%%%%%%%%%%%%%%%%%%%%%%%%%%%%%%%%%%%%%%%%%%%%%%%%%%%%%%%%%%%%%%%%%%%%%%%%%%%%%%%%%

%%%%%%%%%%%%%%%%%%%%%%%%%%%%%%%%%%%%%%%%%%%%%%%%%%%%%%%%%%%%%%%%%%%%%%%%%%%%%%%%%%%%%%%%%%%%%%%%%%%%%%%%%%%%%%%%%%%%%%%%%%%%%%%%%%%%%%%%%%%%%%%%%%%
%%%%%%%%%%%%%%%%%%%%%%%%%%%%%%%%%%%%%%%%%%%%%%%%%%%%%%%%%%%%%%%%%%%%%%%%%%%%%%%%%%%%%%%%%%%%%%%%%%%%%%%%%%%%%%%%%%%%%%%%%%%%%%%%%%%%%%%%%%%%%%%%%%%
%%%%%%%%%%%%%%%%%%%%%%%%%%%%%%%%%%%%%%%%%%%%%%%%%%%%%%%%%%%%%%%%%%%%%%%%%%%%%%%%%%%%%%%%%%%%%%%%%%%%%%%%%%%%%%%%%%%%%%%%%%%%%%%%%%%%%%%%%%%%%%%%%%%

\end{document}